\documentclass[11pt]{amsart}

\usepackage{amsmath}
\usepackage{fullpage}
\usepackage{xspace}
\usepackage[psamsfonts]{amssymb}
\usepackage[latin1]{inputenc}
\usepackage{graphicx,color}
\usepackage[curve]{xypic} 
\usepackage{hyperref}
\usepackage{graphicx}

\usepackage{amsmath}%
\usepackage{amsthm}%
\usepackage{amscd}
\usepackage{amsfonts}%
\usepackage{amssymb}%
\usepackage{graphicx}

\usepackage{mathrsfs}

\usepackage{tikz}
\usetikzlibrary{matrix,arrows}

\usepackage{tikz-cd}

\newtheorem{theorem}{Theorem}[section]

\newtheorem{conjecture}[theorem]{Conjecture}

\newtheorem{corollary}[theorem]{Corollary}

\theoremstyle{remark}

\numberwithin{equation}{section}

\newcommand{\Pro}{\mathbb{P}}

\newcommand{\C}{\mathbb{C}}

\newcommand{\Q}{\mathbb{Q}}

\newcommand{\rk}{\mathrm{rank}\,}

\newcommand{\PGL}{\mathrm{PGL}}

\newcommand{\alg}{\mathrm{alg}}

\makeatletter
\@namedef{subjclassname@2020}{%
  \textup{2020} Mathematics Subject Classification}
\makeatother

  \DeclareFontFamily{U}{wncy}{}
    \DeclareFontShape{U}{wncy}{m}{n}{<->wncyr10}{}
    \DeclareSymbolFont{mcy}{U}{wncy}{m}{n}
    \DeclareMathSymbol{\Sha}{\mathord}{mcy}{"58}

\begin{document}
\title[Patterns on elliptic curves]{Patterns on elliptic curves beyond Bremner's conjecture}

\author{Natalia Garcia-Fritz}
\address{ Departamento de Matem\'aticas,
Pontificia Universidad Cat\'olica de Chile.
Facultad de Matem\'aticas,
4860 Av.\ Vicu\~na Mackenna,
Macul, RM, Chile}
\email[N. Garcia-Fritz]{natalia.garcia@uc.cl}%

\author{Hector Pasten}
\address{ Departamento de Matem\'aticas,
Pontificia Universidad Cat\'olica de Chile.
Facultad de Matem\'aticas,
4860 Av.\ Vicu\~na Mackenna,
Macul, RM, Chile}
\email[H. Pasten]{hector.pasten@uc.cl}%

\thanks{N.G.-F. was supported by ANID Fondecyt Regular grant 1251300 from Chile. H.P. was supported by ANID Fondecyt Regular grant 1230507 from Chile.}
\date{\today}

\begin{abstract} In the late 1990's, Bremner conjectured that long arithmetic progressions among the $x$-coordinates of rational points of an elliptic curve $E$ over $\mathbb{Q}$ should force the rank of $E$ to be large. This conjecture (and a broad generalization of it) was proved by the authors two decades later, by combining Nevanlinna theory and the Uniform Mordell--Lang theorem of Gao--Ge--K\"uhne. The proof inspired subsequent work by the authors where a generalization of the Bogomolov--Fu--Tschinkel conjecture was proved by similar means. In this note we isolate a flexible pattern principle implicit in the latter work, obtaining rank-dependent (but otherwise uniform) bounds for more general patterns in the image of finite rank subgroups of elliptic curves under maps to the projective line. These patterns include, for instance, arithmetic progressions, geometric progressions, additive shifts, multiplicative shifts, and M\"obius orbits.
\end{abstract}

\maketitle



\section{Introduction}

\subsection{Bremner's conjecture} In 1999, Bremner \cite{Bremner} conjectured that long arithmetic progressions in the $x$-coordinates of rational points of an elliptic curve $E$ over $\Q$ with a given Weierstrass equation, should force the rank of $E$ to be large. Beyond concrete numerical examples, some theoretical evidence was obtained in \cite{BST}. This conjecture was proved by the authors in Theorem 6.1 of \cite{GFP}:
\begin{theorem}[Proof of Bremner's conjecture] Let $d$ be a positive integer. There is a constant $c(d)>1$ depending only on $d$ such that the following holds: 

Let $E$ be an elliptic curve over $\C$ and $\Gamma$ a finite rank subgroup of $E(\C)$. Let $\phi:E\to \Pro^1$ be a degree $d$ morphism, seen as a rational function on $E$. If there are $P_1,...,P_n\in \Gamma$ not poles of $\phi$ such that $\phi(P_1),...,\phi(P_n)$ is a non-constant arithmetic progression in $\C$, then 
$$
n\le c(d)^{1+\rk \Gamma}.
$$
\end{theorem}
Bremner's conjecture in its classical formulation is the case when $d=2$, $E$ is defined over $\Q$, $\Gamma=E(\Q)$, and $\phi$ is the $x$-coordinate map.

To be precise, the result in \cite{GFP} used a theorem of R\'emond \cite{Remond1, Remond2} that introduced a dependence on the $j$-invariant of $E$ and required $E$ to be defined over $\Q^{\alg}$. However, as in all applications of R\'emond's bound, this dependence was removed when the Uniform Mordell--Lang Conjecture was proved by Gao--Ge--K\"uhne in 2021 \cite{GGK} (over $\C$, not just $\Q^{\alg}$). See \cite{GFPunif} for a more detailed account.

\subsection{Other patterns} Our purpose is to prove extensions of Bremner's conjecture to other patterns, not just arithmetic progressions. The main point of this note is to isolate a useful pattern principle implicit in our work \cite{GFPinter} on the Bogomolov--Fu--Tschinkel conjecture, namely, Theorem \ref{ThmGeneralV1}. While the arguments here are short, the outcome goes far beyond Bremner's conjecture and it seems worth making it explicit in the literature (see Section \ref{SecResults}). For instance, we will obtain:
\begin{theorem}[M\"obius recurrences]\label{ThmIntro} Let $d$ be a positive integer. There is a constant $c(d)>1$ depending only on $d$ such that the following holds: 

Let $E$ be an elliptic curve over a number field $k$,  let $g:E\to \Pro^1$ be a non-constant morphism over $k$ of degree at most $d$, and let $F:\Pro^1\to \Pro^1$ be an automorphism over $k$ with no iterate equal to the identity (that is, $F\in \PGL_2(k)$ is not a torsion element). If $P_1,P_2,...,P_n\in E(k)$ are pairwise distinct points that satisfy $g(P_{j+1})=F(g(P_j))$ for each $j=1,2,...,n-1$, then 
$$
n\le c(d)^{1+\rk E(k)}.
$$
\end{theorem}
Note that Bremner's conjecture follows from the special case when $k=\Q$ and $F$ has the form $F(t)=t+a$ for non-zero values of $a\in \Q$. On the other hand, taking $F(t)=qt$ for $q\in \Q-\{-1,0,1\}$ gives the analogue of Bremner's conjecture for geometric progressions (see \cite{BremnerUlas, CissMoody} for references on this problem), and so forth. Thus, while this theorem is not the most general result in this note, it already gives a broad generalization of Bremner's conjecture.

\subsection{Shifts} \label{SecIntroShifts} Beyond recurrences, there is another kind of pattern that we are able to study: \emph{shifts}. For the sake of exposition here we just state the case of elliptic curves over $\Q$ and the general case (which includes number fields) will be discussed in Section \ref{SecResultsShifts}.

\begin{theorem}[Additive shifts]\label{ThmAdditive} Let $d$ be a positive integer. There is a constant $\kappa(d)>1$ depending only on $d$ such that the following holds:

Let $E$ be an elliptic curve over $\Q$. Let $g:E\to \Pro^1$ be a non-constant morphism over $\Q$ of degree $\le d$ seen as a rational function on $E$. Let $a\in \Q^\times$. If $S\subseteq g(E(\Q))$ and $S+a\subseteq g(E(\Q))$, then
$$
\# S\le \kappa(d)^{1+\rk E(\Q)}.
$$
\end{theorem}
\begin{theorem}[Multiplicative shifts]\label{ThmMultiplicative} Let $d$ be a positive integer. There is a constant $\kappa(d)>1$ depending only on $d$ such that the following holds:

Let $E$ be an elliptic curve over $\Q$. Let $g:E\to \Pro^1$ be a non-constant morphism over $\Q$ of degree $\le d$ seen as a rational function on $E$. Let $q\in \Q-\{-1,0,1\}$. If $S\subseteq g(E(\Q))$ and $q\cdot S\subseteq g(E(\Q))$, then
$$
\# S\le \kappa(d)^{1+\rk E(\Q)}.
$$
\end{theorem}
Here is an example of the shift phenomenon in the additive and multiplicative setting, which is not merely a single arithmetic or geometric progression:

\begin{proof}[Example] Consider the elliptic curve
$$
E:\quad y^2+xy=x^3-x^2-79x+289
$$
with LMFDB label 234446.a1 \cite{LMFDB}. The following values occur as $x$-coordinates of rational points on $E$:
$$
-10,-9,-8,-7,-4,0,1,3,4,5,6,7,8,12,13.
$$
Hence, for
$$
S=\{-10,-9,-8,0,3,4,5,6,7,12\},
$$
we have
$$
S+1=\{-9,-8,-7,1,4,5,6,7,8,13\} \subseteq x(E(\Q)).
$$
Thus $S\subseteq x(E(\Q))$ and $S+1\subseteq x(E(\Q))$, with $\#S=10$.

The same curve also gives a multiplicative-shift example. Namely,
$$
S=\{-4,3,4,6\}
$$
satisfies
$$
2\cdot S=\{-8,6,8,12\}\subseteq x(E(\mathbb Q)).
$$
It turns out that there is a remarkable aspect of this elliptic curve regarding its rank: it is the elliptic curve of smallest conductor having rank $4$.
\end{proof}

\subsection{Non-linear recurrences: a motivating example} Our results in Section \ref{SecResults} provide bounds for more general patterns, such as the one in the following example:

\begin{proof}[Example] Consider the elliptic curve
$$
E:\quad y^2+y=x^3-7x+6
$$
with LMFDB label 5077.a1 \cite{LMFDB}. Define the quadratic polynomial
$$
    F(t)=-\frac{1}{6}t^2-\frac{7}{6}t+2.
$$
Then we have the $F$-orbit segment
$$
0,\ 2,\ -1,\ 3,\ -3,\ 4
$$
namely, 
$$
F(0)=2,\quad
F(2)=-1,\quad
F(-1)=3,\quad
F(3)=-3,\quad
F(-3)=4.
$$
These $6$ numbers occur in $x(E(\Q))$, as we have the rational points
$$
\begin{array}{c|rrrrrr}
x & 0 & 2 & -1 & 3 & -3 & 4 \\ \hline
y & 2 & 0 & 3 & 3 & 0 & 6
\end{array}
$$
It is worth pointing out that $E$ is the elliptic curve with smallest conductor having rank $3$.
\end{proof}

In simple terms the general principle will be that, unless a pattern is directly related to the group structure of the elliptic curve, \emph{a long pattern implies large rank}.

\subsection{Outline of the rest of the paper} The rest of the article is organized as follows: first we will recall our work on the Bogomolov--Fu--Tschinkel conjecture \cite{GFPinter} and then we will remark that Bremner's conjecture, as well as its multiplicative analogue for geometric progressions, are just special cases of that work by choosing appropriate parameters. This will serve as a motivation for the other results of the present work, discussed in Section \ref{SecResults}. The main result is Theorem \ref{ThmGeneralV1}, that we call \emph{the pattern principle} as it is the underlying tool for all the other pattern bounds in this article.


\section{From Bogomolov--Fu--Tschinkel to Bremner}

\subsection{Intersecting values of rational functions on elliptic curves} In a series of works \cite{BF,BFT,BT}, Bogomolov, Fu, and Tschinkel formulated the following conjecture: 

\begin{conjecture} There is a constant $c$ with the following property: If $E_1$ and $E_2$ are elliptic curves defined over $\C$ given by Weierstrass equations $y^2=f_1(x)$ and $y^2=f_2(x)$ where $f_1$ and $f_2$ do not have the same roots, then
$$
\#x(E_1(\C)_{\rm tor})\cap x(E_2(\C)_{\rm tor}) \le c.
$$
\end{conjecture}

After an initial breakthrough in \cite{DMKY}, this conjecture was fully proved due to the proof of the Uniform Manin--Mumford Conjecture as noted in \cite{FuStoll}. We refer the reader to \cite{GFPinter} for a more detailed discussion of this problem. In the same paper \cite{GFPinter} we proved a generalization where the $x$-coordinate maps can be replaced by other rational functions, namely:

\begin{theorem} Let $d$ be a positive integer. There is a constant $c(d)>1$ such that the following holds:

If $E_1$ and $E_2$ are elliptic curves defined over $\C$ and $g_j:E_j\to \Pro^1$ are non-constant morphisms of degree $\le d$ that do not have the same set of branch values, then
$$
\#\left(g_1(E_1(\C)_{\rm tor})\cap g_2(E_2(\C)_{\rm tor})\right) \le c(d).
$$
\end{theorem}

In \cite{GFPinter} we also prove an arithmetic counterpart:

\begin{theorem}\label{ThmInterk} Let $d$ be a positive integer. There is a constant $c(d)>1$ such that the following holds:

Let $k$ be a number field. If $E_1$ and $E_2$ are elliptic curves defined over $k$ and $g_j:E_j\to \Pro^1$ are non-constant morphisms of degree $\le d$ defined over $k$ that do not have the same set of (complex) branch values, then
$$
\#\left(g_1(E_1(k))\cap g_2(E_2(k))\right) \le c(d)^{1 + \rk E_1(k) + \rk E_2(k)}.
$$
\end{theorem}

\subsection{Arithmetic and geometric progressions} Let us consider the following special cases of Theorem \ref{ThmInterk}: take $k=\Q$,  $E:=E_1=E_2$ any elliptic curve over $\Q$, and choosing a Weierstrass equation for $E$ let $g_1=x$ be the $x$-coordinate map. Let $d=2$. Then take any of the following two choices of $g_2$:
\begin{itemize}
\item[(i)] For $a\in \Q^\times$ let $g_2 = x+a$
\item[(ii)] For $q\in \Q-\{0,-1,1\}$ take $g_2= qx$.
\end{itemize}

In case (i) we deduce Bremner's conjecture at once: a sequence $P_1,...,P_n\in E(\Q)$ with $x(P_j)= aj+b$ for some $b\in \Q$ satisfies that 
$$
x(P_{j+1}) = x(P_j)+a \in g_1(E(\Q))\cap g_2(E(\Q)) \quad \mbox{ for each }j=1,2,...,n-1. 
$$
Thus, $n \le 1+ c(2)^{1+ 2\rk E(\Q)}$. So, a long arithmetic progression on the $x$-coordinates of $E(\Q)$ forces $\rk E(\Q)$ to be large.

Case (ii) immediately gives the analogue of Bremner's conjecture for geometric progressions: any sequence $P_1,...,P_n\in E(\Q)$ with $x(P_j)= pq^j$ for some $p\in \Q^\times$ satisfies 
$$
x(P_{j+1}) = q\cdot x(P_j)\in g_1(E(\Q))\cap g_2(E(\Q)) \quad \mbox{ for each }j=1,2,...,n-1. 
$$
Thus, $n \le 1+ c(2)^{1+ 2\rk E(\Q)}$; hence, a long geometric progression in $x(E(\Q))$ forces $\rk E(\Q)$ to be large.

Of course there is nothing special about the $x$-coordinate map; any non-constant rational function on $E$ works as long as we keep the degree bounded, e.g. $y$-coordinates. After these examples, the proof of Theorem \ref{ThmIntro} is an exercise (in any case, it will follow from the more general results of Section \ref{SecResults}).

The special case (i) is not surprising. After all, our arguments in \cite{GFPinter} are modeled after our proof of Bremner's conjecture \cite{GFP}: Nevanlinna theory plus Uniform Mordell--Lang. Thus, this should be viewed as the same underlying method rather than as a genuinely different proof.

Regarding the case (ii) of geometric progressions, this particular instance of Theorem \ref{ThmInterk} was recently proved independently by Harrison--Mudgal--Schmidt in \cite{HMS}, using rather different methods. In the present framework it is a specialization of Theorem \ref{ThmInterk}, as explained above. We note that the applications in \cite{HMS} are not restricted to geometric progressions and also include an alternative proof of Bremner's conjecture; the techniques developed there are of independent interest.

\subsection{The case of finite rank subgroups} In \cite{GFPinter} we presented the main results specialized to two extremal cases: the torsion subgroup of a complex elliptic curve, and the Mordell--Weil group of an elliptic curve over a number field. Both cases had the same proof and, in fact, if instead of choosing a particular finite rank subgroup one makes no choice and simply keeps that level of generality (using the Uniform Mordell--Lang theorem of Gao--Ge--K\"uhne \cite{GGK} stated as Theorem 4.1 in \cite{GFPinter}) then one obtains:

\begin{theorem}\label{ThmFiniteRankInter} Let $d$ be a positive integer. There is a constant $c(d)$ such that the following holds:

Let $E_1$ and $E_2$ be elliptic curves defined over $\C$ and $g_j:E_j\to \Pro^1$ non-constant morphisms of degree $\le d$ that do not have the same set of branch values. Let $\Gamma_j\le E_j(\C)$ for $j=1,2$ be finite rank subgroups (not necessarily finitely generated). Then
$$
\#\left(g_1(\Gamma_1)\cap g_2(\Gamma_2)\right) \le c(d)^{1+\rk \Gamma_1 + \rk \Gamma_2}.
$$
\end{theorem}
The proof is the same as the arguments presented in \cite{GFPinter} which, for the convenience of the reader, will be outlined here:

Consider the abelian surface $A=E_1\times E_2$ and the map $f:A\to \Pro^1\times\Pro^1$ given by $f(P,Q)=(g_1(P),g_2(Q))$. Let $C\subseteq A$ be the $1$-dimensional subvariety given by the support of $f^*\Delta$ where $\Delta$ is the diagonal of  $\Pro^1\times\Pro^1$. Let $\Gamma = \Gamma_1\times\Gamma_2$; this is a finite rank subgroup of $A(\C)$. The key observation is that, in order to bound
$$
\#\left(g_1(\Gamma_1)\cap g_2(\Gamma_2)\right)
$$
it suffices to bound 
$$
\#\left(C\cap \Gamma \right).
$$
Such a bound follows at once from the Uniform Mordell--Lang Theorem of Gao--Ge--K\"uhne \cite{GGK} provided that
\begin{itemize}
\item[(i)] One bounds the degree of each component of $C$ with respect to an ample line bundle, and
\item[(ii)] One shows that each irreducible component of $C$ is a curve of geometric genus at least $2$.
\end{itemize}

The first requirement is an intersection-theoretic computation. 

The second one is achieved via Nevanlinna theory of complex holomorphic maps, because a projective complex curve $X$ has geometric genus $0$ or $1$ precisely when there is a non-constant complex holomorphic map $h:\C\to X$. This is the most delicate part of the argument and it is here where the branching hypothesis is used.

With all these elements in place, Theorem \ref{ThmFiniteRankInter} follows.


\section{Main Results}\label{SecResults}

\subsection{The pattern principle} Here is the main result of this note.

\begin{theorem}[The pattern principle]\label{ThmGeneralV1} Let $d$ be a positive integer. There is a constant $\kappa_1(d)>1$ depending only on $d$ such that the following holds:

Let $E$ be an elliptic curve over $\C$. Let $g:E\to \Pro^1$ and $F:\Pro^1\to\Pro^1$ be non-constant morphisms of degree $\le d$ such that $g$ and $F\circ g$ do not have the same set of branch values. Let $\Gamma$ be a finite rank subgroup of $E(\C)$. If $S\subseteq g(\Gamma)$ and $F(S)\subseteq g(\Gamma)$, then
$$
\# S\le \kappa_1(d)^{1+\rk \Gamma}.
$$
\end{theorem}
\begin{proof} This follows from Theorem \ref{ThmFiniteRankInter} by taking $E_1=E_2=E$, $\Gamma_1=\Gamma_2=\Gamma$, $g_1=g$, $g_2=F\circ g$ (which has degree $\le d^2$), and observing that 
$$
F(S) \subseteq  g_1(\Gamma)\cap g_2(\Gamma).
$$
From here the desired bound follows from $\# F(S)\ge (\#S)/d$. 
\end{proof}

Before any further analysis, let us immediately point out that the branching hypothesis is necessary. For instance, one could take $g:E\to \Pro^1$ as the $x$-coordinate map for a fixed Weierstrass equation and $F$ the corresponding Latt\`es map for duplication on $E$. Then $F(x(\Gamma)) = x([2](\Gamma))\subseteq x(\Gamma)$.

As an application of the pattern principle, one has control on the patterns that arise from certain recurrences.

\begin{theorem}[Bounding patterns: recurrences]\label{ThmGeneralV2} Let $d$ be a positive integer. There is a constant $\kappa_2(d)>1$ depending only on $d$ such that the following holds:

Let $E$ be an elliptic curve over $\C$. Let $g:E\to \Pro^1$ and $F:\Pro^1\to\Pro^1$ be non-constant morphisms of degree $\le d$ such that $g$ and $F\circ g$ do not have the same set of branch values. Let $\Gamma$ be a finite rank subgroup of $E(\C)$. Let $n\ge 1$, $\alpha_1\in \Pro^1(\C)$, and for each $j=1,...,n-1$ define $\alpha_{j+1} = F(\alpha_j)$. Suppose that all the $\alpha_j$ are distinct and belong to $g(\Gamma)$. Then 
$$
n\le \kappa_2(d)^{1+\rk \Gamma}.
$$
\end{theorem}
\begin{proof} Take $S=\{\alpha_1,...,\alpha_{n-1}\}$ in Theorem \ref{ThmGeneralV1}.
\end{proof}

Note that these results are, in particular, applicable to the case $\Gamma = E(\C)_{\rm tor}$ (here the rank is $0$ so the bounds become uniform) as well as when everything is defined over a number field and $\Gamma$ is the Mordell--Weil group of the elliptic curve.

\subsection{The case of M\"obius transformations}  There is a case of particular interest where the branching hypothesis is easily seen to be satisfied: when $F$ is an automorphism of $\Pro^1$ of infinite order in $\PGL_2(\C)$. Indeed, any non-constant map $g:E\to \Pro^1$ is branched on at least $3$ values in $\Pro^1$, and an $F\in \PGL_2(\C)$ that permutes them has finite order. Let us state the corresponding two consequences in this setting.

\begin{theorem}[The M\"obius case: invariant sets]\label{ThmMobiusV1} Let $d$ be a positive integer. There is a constant $\kappa_1(d)>1$ depending only on $d$ such that the following holds:

Let $E$ be an elliptic curve over $\C$. Let $g:E\to \Pro^1$ be a non-constant morphism of degree $\le d$ and let $F \in \PGL_2(\C)$ be an automorphism of $\Pro^1$ of infinite order.  Let $\Gamma$ be a finite rank subgroup of $E(\C)$. If $S\subseteq g(\Gamma)$ and $F(S)\subseteq g(\Gamma)$, then
$$
\# S\le \kappa_1(d)^{1+\rk \Gamma}.
$$
\end{theorem}
\begin{theorem}[The M\"obius case: orbits]\label{ThmMobiusV2} Let $d$ be a positive integer. There is a constant $\kappa_2(d)>1$ depending only on $d$ such that the following holds:

Let $E$ be an elliptic curve over $\C$. Let $g:E\to \Pro^1$ be a non-constant morphism of degree $\le d$ and let $F \in \PGL_2(\C)$ be an automorphism of $\Pro^1$ of infinite order.  Let $\Gamma$ be a finite rank subgroup of $E(\C)$. Let $n\ge 1$, $\alpha_1\in \Pro^1(\C)$, and for each $j=1,...,n-1$ define $\alpha_{j+1} = F(\alpha_j)$. Suppose that all the $\alpha_j$ are distinct and belong to $g(\Gamma)$. Then 
$$
n\le \kappa_2(d)^{1+\rk \Gamma}.
$$
\end{theorem}
Note that Theorem \ref{ThmIntro} is a special case of Theorem \ref{ThmMobiusV2}. 
\subsection{Addition and multiplication}\label{SecResultsShifts} By a suitable choice of $F$ in Theorem \ref{ThmMobiusV2} we recover Bremner's conjecture for arithmetic progressions and geometric progressions in a very general form. But actually Theorem \ref{ThmMobiusV1} gives more:

\begin{corollary}[Additive shifts]\label{CoroAdditive} Let $d$ be a positive integer. There is a constant $\kappa_1(d)>1$ depending only on $d$ such that the following holds:

Let $E$ be an elliptic curve over $\C$. Let $g:E\to \Pro^1$ be a non-constant morphism of degree $\le d$ seen as a rational function on $E$. Let $a\in \C^\times$.   Let $\Gamma$ be a finite rank subgroup of $E(\C)$. If $S\subseteq g(\Gamma)$ and $S+a\subseteq g(\Gamma)$, then
$$
\# S\le \kappa_1(d)^{1+\rk \Gamma}.
$$
\end{corollary}
\begin{corollary}[Multiplicative shifts]\label{CoroMultiplicative} Let $d$ be a positive integer. There is a constant $\kappa_1(d)>1$ depending only on $d$ such that the following holds:

Let $E$ be an elliptic curve over $\C$. Let $g:E\to \Pro^1$ be a non-constant morphism of degree $\le d$ seen as a rational function on $E$. Let $q\in \C^\times$ be a complex number which is not a root of unity. Let $\Gamma$ be a finite rank subgroup of $E(\C)$. If $S\subseteq g(\Gamma)$ and $q\cdot S\subseteq g(\Gamma)$, then
$$
\# S\le \kappa_1(d)^{1+\rk \Gamma}.
$$
\end{corollary}
As it might be of particular interest, at this point we insist that our results are valid in the special cases $\Gamma = E(\C)_{\rm tor}$ and when $\Gamma$ is the Mordell--Weil group of an elliptic curve over a number field. The latter case over $\Q$ precisely gives the theorems stated in Section \ref{SecIntroShifts}.


\section{Acknowledgments}

N.G.-F. was supported by ANID Fondecyt Regular grant 1251300 from Chile. H.P. was supported by ANID Fondecyt Regular grant 1230507 from Chile. 



\begin{thebibliography}{9}         


\bibitem{BF} F. Bogomolov, H. Fu, \emph{Division polynomials and intersection of projective torsion points}. Eur. J. Math. 2(3) (2016), 644-660.

\bibitem{BFT} F. Bogomolov, H. Fu, Y. Tschinkel, \emph{Torsion of elliptic curves and unlikely intersections}.
in: Geometry and Physics, Vol. I (eds. J. E. Andersen, A. Dancer and O. Garc\'ia-Prada) (Oxford University Press, Oxford, 2018), 19-37.

\bibitem{BT} F. Bogomolov, Y. Tschinkel, \emph{Algebraic varieties over small fields}. in: Diophantine Geometry, CRM Series, 4 (ed. U. Zannier) (Scuola Normale Superiore di Pisa, Pisa, 2007), 73-91.

\bibitem{Bremner} A. Bremner, \emph{On arithmetic progressions on elliptic curves}. Experimental Mathematics, 1999, vol. 8, no 4, p. 409-413.

\bibitem{BST} A. Bremner, J. Silverman, N. Tzanakis, \emph{Integral points in arithmetic progression on $y^2=x(x^2-n^2)$}. Journal of Number Theory, (2000) 80(2), 187-208.

\bibitem{BremnerUlas} A. Bremner, M. Ulas, \emph{Rational points in geometric progressions on certain hyperelliptic curves}. Publ. Math. Debrecen 82.3--4 (2013): 669-683.

\bibitem{CissMoody} A. Ciss, D. Moody, \emph{Geometric progressions on elliptic curves}. Glasnik matematicki 52.1 (2017): 1-10.

\bibitem{DMKY} L. DeMarco, H. Krieger, H. Ye, \emph{Uniform Manin-Mumford for a family of genus 2 curves}. Ann. of Math. (2) 191 (2020), no. 3, 949-1001.

\bibitem{FuStoll} H. Fu, M. Stoll, \emph{Elliptic curves with common torsion $x$-coordinates and hyperelliptic torsion packets}. Proc. Amer. Math. Soc. 150 (2022), no. 12, 5137-5149.

\bibitem{GFP} N. Garcia-Fritz, H. Pasten, \emph{Elliptic curves with long arithmetic progressions have large rank}. Int. Math. Res. Not. IMRN 2021, no. 10, 7394-7432.

\bibitem{GFPinter} N. Garcia-Fritz, H. Pasten, \emph{Intersecting the torsion of elliptic curves}. Bull. Aust. Math. Soc. 110 (2024), no. 1, 56-63.

\bibitem{GFPunif} N. Garcia-Fritz, H. Pasten, \emph{A note on Bremner's conjecture and uniformity}. Preprint (2026) arXiv:2604.04850

\bibitem{GGK} Z. Gao, T. Ge, L. K\"uhne, \emph{The Uniform Mordell-Lang Conjecture}. (2021) to appear in Publ. Math. IHES.

\bibitem{HMS} J. Harrison, A. Mudgal, H. Schmidt, \emph{Uniform sum-product phenomenon for algebraic groups and Bremner's conjecture}. Preprint (2026) arXiv:2603.06483

\bibitem{LMFDB} The LMFDB Collaboration, \emph{The $L$-functions and modular forms database}. \url{https://www.lmfdb.org} (2026).

\bibitem{Remond1} G. R\'emond, \emph{D\'ecompte dans une conjecture de Lang}. Inventiones Mathematicae, (2000) 142 (3), 513-545.

\bibitem{Remond2} G. R\'emond, \emph{Sur les sous-vari\'et\'es des tores}. Compositio Mathematica 134.3 (2002) 337-366.

\end{thebibliography}
\end{document}